\documentclass[11pt]{amsart}

\usepackage[margin=1in]{geometry}
\usepackage{amsmath,amssymb,mathtools}
\usepackage{microtype}
\usepackage{xcolor}
\usepackage{hyperref}
\usepackage{placeins}
\usepackage{tikz}
\usetikzlibrary{
  arrows.meta,
  calc,
  decorations.pathreplacing,
  fit,
  matrix,
  positioning
}

\hypersetup{
  colorlinks=true,
  linkcolor=blue,
  citecolor=blue,
  urlcolor=blue
}

\newtheorem{theorem}{Theorem}[section]
\newtheorem{proposition}[theorem]{Proposition}
\newtheorem{corollary}[theorem]{Corollary}
\newtheorem{lemma}[theorem]{Lemma}
\newtheorem*{auxlemma}{Lemma}

\theoremstyle{definition}
\newtheorem{definition}[theorem]{Definition}
\newtheorem{example}[theorem]{Example}

\theoremstyle{remark}
\newtheorem{remark}[theorem]{Remark}

\numberwithin{equation}{section}

\newcommand{\F}{\mathbb F}
\newcommand{\pinf}{\mathfrak p}
\newcommand{\cK}{\mathcal K}
\newcommand{\cA}{\mathcal A}
\newcommand{\coef}{\boldsymbol\theta}
\newcommand{\fp}[1]{\left\{#1\right\}}
\newcommand{\pp}[1]{\left[#1\right]}
\newcommand{\norm}[1]{\left\lVert #1\right\rVert}
\newcommand{\Prob}{\mathbb P}
\newcommand{\E}{\mathbb E}
\newcommand{\one}{\mathbf 1}
\newcommand{\JL}{\mathcal L}
\newcommand{\Nzero}{\mathbb N_0}

\definecolor{profileblue}{RGB}{42,101,171}
\definecolor{freshred}{RGB}{190,55,55}
\definecolor{survivegreen}{RGB}{55,135,90}
\definecolor{knownfill}{RGB}{229,236,246}
\definecolor{freshfill}{RGB}{249,220,217}

\newcommand{\subheading}[1]{\par\medskip\noindent{\bfseries #1.}\quad}
\makeatletter
\renewcommand{\@bibtitlestyle}{\par\bigskip\begin{center}{\scshape\refname}\end{center}\medskip}
\makeatother

\title[Exact renewal laws for minimal common-denominator profiles]
{Exact Renewal Laws for Minimal Common-Denominator Profiles in
Simultaneous Laurent-Series Approximation}

\author{Sanghoon Kwon}
\address{Department of Mathematics Education\\ Catholic Kwandong University \\ Gangneung 25601 \\ Republic of Korea}
\email{shkwon1988@gmail.com, skwon@cku.ac.kr}

\date{August 6, 2026}

\begin{document}

\begin{abstract}
Let $\alpha_1,\ldots,\alpha_r$ be independent Haar-random fractional
Laurent series over $\F_q$, and let $L_r(n)$ be the least coefficient
length of a polynomial denominator that simultaneously cancels the
first $n$ negative coefficients.  We prove that the minimal kernel is
a line and that the residual vectors revealed immediately after the
stopping times $T_n=n+L_r(n)-1$ are iid uniform on $\F_q^r$.  Hence the jump
indicators of $L_r(n)$ are iid Bernoulli variables with parameter
$1-q^{-r}$; conditionally on a jump, the residual direction is uniform
on $\mathbb P^{r-1}(\F_q)$.  We also give an exact kernel-growth clock
for positive jump sizes and a geometric tail bound uniform in the depth;
for two series the jump is decided at the first or second kernel-growth
epoch with probabilities $q^{-1}$ and $1-q^{-1}$.  The marked renewal
law yields exact
binomial and fluctuation laws in the depth variable and the density
of newly attained minimal denominator lengths
\[
        \frac{1-q^{-r}}{r}
\]
in the coefficient-length variable.  For $r=1$ this is the classical
iid partial-quotient degree law in the depth coordinate, for which we
give an exact dictionary.  The new probabilistic content is the
simultaneous common-denominator law for $r\ge2$.  We also establish
exact profile-correspondence and record-duality formulas with joint linear
complexity.
\end{abstract}

\maketitle

\setcounter{tocdepth}{1}
\tableofcontents

\section{Introduction}

Throughout, $q$ is a prime power, $\F_q$ is the field with $q$
elements, and $r\ge1$ is an integer.  Let
\[
        K=\F_q(\!(t^{-1})\!)
\]
be the Laurent-series field and let
\[
        \pinf=t^{-1}\F_q[\![t^{-1}]\!]
\]
be its compact additive subgroup of fractional Laurent series.  Every
$\alpha_i\in\pinf$ has a unique expansion
\[
        \alpha_i=\sum_{j\ge1}u^{(i)}_jt^{-j},
        \qquad 1\le i\le r,
\]
with $u_j^{(i)}\in\F_q$.  We equip $\pinf^r$ with product Haar
probability measure.  Equivalently, the coefficient vectors
\[
        \boldsymbol u_j
        =
        (u_j^{(1)},\ldots,u_j^{(r)})\in\F_q^r,
        \qquad j\ge1,
\]
are independent and identically distributed (iid), each uniformly
distributed on $\F_q^r$.

We write $\Prob$ and $\E$ for probability and expectation,
$\Nzero=\{0,1,2,\ldots\}$, $\#E$ for the cardinality of a finite set
$E$, and $\one_E$ for the indicator of an event $E$.  For positive
sequences, $a_N\sim b_N$ means $a_N/b_N\to1$.
For $x=\sum_j a_jt^j\in K$, let $[t^m]x=a_m$.  Its cancellation
length is
\[
 c(x)=\max\{n\ge0:[t^{-1}]x=\cdots=[t^{-n}]x=0\},
\]
with $c(x)=\infty$ when all negative coefficients vanish.  Finally,
\[
 \F_q^\times=\F_q\setminus\{0\},\qquad
 \mathbb P^{r-1}(\F_q)
 =
 (\F_q^r\setminus\{0\})/\F_q^\times,
\]
and $[v]$ denotes the projective class of a nonzero vector
$v\in\F_q^r$.

For $\alpha_1,\ldots,\alpha_r\in\pinf$,
we ask for a single nonzero polynomial $Q\in\F_q[t]$ such that the
first $n$ negative coefficients of every $Q\alpha_i$ vanish.  If
\[
        d(Q)=\deg Q+1
\]
is the coefficient length of $Q$, define
\[
        L_r(n)
        =
        \min\{d(Q):Q\neq0,\ Q\alpha_i
        \text{ has cancellation depth at least }n
        \text{ for every }i\}.
\]
We call $n\mapsto L_r(n)$ the \emph{minimal common-denominator
length profile}.  It is a nondecreasing staircase.

The one-series problem belongs to the continued-fraction theory of
Laurent series, initiated by Artin and developed in many later
directions; see \cite{ArtinI,ArtinII,deMathan,Lasjaunias,Rosen,
Schmidt}.  Its metric theory includes the positive-characteristic
Khinchin law of Houndonougbo and Berth\'e--Nakada, the exactness of
the Artin continued-fraction map, explicit cylinder measures, and
quantitative laws for partial quotients
\cite{Houndonougbo,BertheNakada,LertchoosakulNairMetric,
LertchoosakulNairQuantitative}.  From the linear-complexity side,
Niederreiter's probabilistic theory and Vielhaber's continued-fraction
isometry give Bernoulli codings of the one-series profile
\cite{NiederreiterProbabilistic,NiederreiterCombinatorial,
VielhaberIsometry}.  Related central limit and invariance principles
for counts of Diophantine inequalities were obtained in
\cite{FuchsMetric,DeligeroNakada,DeligeroFuchsNakada}; those papers
concern different counting observables from the profile-value count
$C_N^{(r)}$ below.  More recent fine-scale metric questions for
partial-quotient degrees are studied in \cite{HuHussainYu}.
For several series the same coefficient equations occur in joint
linear complexity, multisequence synthesis, multi-continued fractions,
and simultaneous Pad\'e approximation
\cite{DaiFengYang,NiederreiterJointSurvey,NiederreiterWang,
VielhaberCanales,BeckermannLabahn}.

The present paper isolates a particular observable that is not the
full joint-linear-complexity profile.  At a fixed depth $n$ we choose
the unique monic generator $Q_n$ of the kernel at the least possible
coefficient length $L_r(n)$.  The first residual vector not already
used to determine $Q_n$ decides whether this same minimal kernel line
persists at depth $n+1$.  Write
\[
 \rho_n^{(r)}
 =
 \bigl(
 [t^{-(n+1)}]Q_n\alpha_1,\ldots,
 [t^{-(n+1)}]Q_n\alpha_r
 \bigr)\in\F_q^r.
\]

\begin{theorem}[Exact marked renewal at the stopping times]
\label{thm:intro-renewal}
Let $\alpha_1,\ldots,\alpha_r$ be independent Haar-random elements of
$\pinf$.  Then the vectors
\[
        \rho_0^{(r)},\rho_1^{(r)},\ldots
\]
are iid and uniformly distributed on $\F_q^r$.  Moreover,
\[
        J_n^{(r)}
        =
        \one_{\{L_r(n+1)>L_r(n)\}}
        =
        \one_{\{\rho_n^{(r)}\ne0\}},
        \qquad n\ge0.
\]
Consequently, the variables $J_0^{(r)},J_1^{(r)},\ldots$ are iid
Bernoulli variables with
\[
        \Prob(J_n^{(r)}=1)=1-q^{-r}.
\]
Conditionally on $J_n^{(r)}=1$, the projective mark
$[\rho_n^{(r)}]\in\mathbb P^{r-1}(\F_q)$ is uniform.
\end{theorem}

Writing
\[
 \Delta_n^{(r)}=L_r(n+1)-L_r(n),
 \qquad
 \mathcal S_n^{(r)}=rn+1-L_r(n),
\]
Theorem~\ref{thm:jump-clock} below gives, for every $h\ge r+1$,
the depth-uniform estimate
\[
 \Prob\bigl(\Delta_n^{(r)}\ge h\mid J_n^{(r)}=1\bigr)
 \le \frac{q^{r+1-h}}{q-1}.
\]
It also identifies the exact kernel-growth epoch at which a positive
jump occurs.  In particular, for $r=2$ the jump is selected from two
successive kernel-growth epochs, with probabilities $q^{-1}$ and
$1-q^{-1}$.

Let
\[
        d_1^{(r)}<d_2^{(r)}<\cdots
\]
be the increasing list of the distinct values of $L_r(n)$.  For the
Haar-random systems considered in the main theorem this list is
infinite almost surely.

\begin{proposition}[Classical one-series dictionary]
\label{prop:one-series}
Let $r=1$ and let
\[
        \alpha=[0;A_1,A_2,\ldots]
\]
be the continued-fraction expansion of a Haar-typical
$\alpha\in\pinf$.  Let $P_j/Q_j$ be its convergents, normalized by
$Q_0=1$, and put
\[
        s_0=0,
        \qquad
        s_j=\deg Q_j=\sum_{k=1}^j\deg A_k
        \quad (j\ge1).
\]
Then, for every $j\ge1$,
\[
        L_1(n)=s_{j-1}+1
        \qquad
        (s_{j-1}\le n\le s_j-1).
\]
Consequently,
\[
        d_j^{(1)}=s_{j-1}+1,
        \qquad
        J_n^{(1)}
        =
        \one_{\{n+1\in\{s_1,s_2,\ldots\}\}}.
\]
The $j$th plateau has length
\[
        s_j-s_{j-1}=\deg A_j.
\]
Moreover, the partial quotients are iid and
\[
 \Prob(\deg A_j=h)
 =
 (q-1)q^{-h}
 =
 (1-q^{-1})q^{-(h-1)},
 \qquad h\ge1.
\]
Thus the case $r=1$ of
Theorem~\ref{thm:intro-renewal}, including its exact finite-time and
fluctuation consequences, is the classical one-series law in the
depth coordinate.
\end{proposition}

\begin{proof}
The standard convergent identities and best-approximation property
give
\[
        c(Q_{j-1}\alpha)=\deg Q_j-1=s_j-1
\]
and, for $j\ge2$,
\[
 0\neq Q,\quad \deg Q<\deg Q_{j-1}
 \quad\Longrightarrow\quad
 c(Q\alpha)\le s_{j-1}-1;
\]
see, for example,
\cite{deMathan,Schmidt,Rosen,VielhaberIsometry}.
Hence $Q_{j-1}$ has the least possible coefficient length at every
depth $s_{j-1}\le n\le s_j-1$.  This proves the formula for $L_1(n)$
and the ensuing identities.

For polynomials $B_1,\ldots,B_m$ of positive degree, let
$\Delta_{B_1,\ldots,B_m}$ be the set of elements of $\pinf$ whose
first $m$ partial quotients are $B_1,\ldots,B_m$.  If $\mu$ denotes
normalized Haar probability measure on $\pinf$, the
continued-fraction cylinder formula is
\[
 \mu(\Delta_{B_1,\ldots,B_m})
 =
 q^{-2\sum_{k=1}^m\deg B_k}
 =
 \prod_{k=1}^m q^{-2\deg B_k};
\]
see \cite{BertheNakada,LertchoosakulNairMetric,
LertchoosakulNairQuantitative}.  It follows that the partial quotients
are iid and that
\[
        \Prob(A_j=B)=q^{-2\deg B}.
\]
There are $(q-1)q^h$ polynomials of degree $h$, which gives the
displayed geometric law.  Equivalently, Vielhaber's degree code is an
iid uniform $\F_q$-sequence whose nonzero-symbol indicators are the
variables $J_n^{(1)}$ above
\cite{NiederreiterProbabilistic,VielhaberIsometry}.
\end{proof}

\begin{theorem}[Attained-length density]
\label{thm:intro-density}
For Haar-almost every
$(\alpha_1,\ldots,\alpha_r)\in\pinf^r$,
\[
 \lim_{D\to\infty}
 \frac1D\#\{j:d_j^{(r)}\le D\}
 =\frac{1-q^{-r}}{r}.
\]
Equivalently,
\[
        \lim_{j\to\infty}\frac{d_j^{(r)}}j
        =
        \frac{r}{1-q^{-r}}
        =
        \frac{rq^r}{q^r-1}.
\]
\end{theorem}

For $r=2$ the density and mean spacing are
\[
        \lambda_{2,q}=\frac{q^2-1}{2q^2},
        \qquad
        \kappa_{2,q}=\frac{2q^2}{q^2-1}.
\]
For $r=1$ the spacing becomes $q/(q-1)$, exactly the mean degree of a
Haar-random Laurent-series partial quotient from
Proposition~\ref{prop:one-series}; this is the classical
positive-characteristic Khinchin law
\cite{Houndonougbo,BertheNakada,LertchoosakulNairQuantitative}.

\subheading{Exact dualities with joint linear complexity}
Let $\JL_r(N)$ denote the least common recurrence length of the first
$N$ coefficient vectors.  For a nonzero polynomial $Q$, let $m_r(Q)$
be its common cancellation depth, and put
\[
 R_D^{(r)}
 =
 \max_{\substack{Q\ne0\\ d(Q)\le D}}
 \bigl(rm_r(Q)-d(Q)\bigr).
\]
The formal definitions, including the zero-prefix convention for
$\JL_r$, are given in Section~\ref{sec:joint}.

\begin{proposition}[Exact profile--complexity dualities]
\label{prop:intro-dualities}
For every $n\ge0$, if
\[
        d=L_r(n),
        \qquad
        N=n+d-1,
\]
then
\[
        \JL_r(N)=d-1.
\]
If no nonzero polynomial has infinite common cancellation depth, then,
for every $D\ge1$,
\[
 R_D^{(r)}
 =
 \max_{\substack{N\in\Nzero\\ \JL_r(N)+1\le D}}
 \bigl(rN-(r+1)\JL_r(N)-1\bigr).
\]
\end{proposition}

The first identity is proved in
Proposition~\ref{prop:profile-correspondence}, and the second in
Proposition~\ref{prop:record-duality}.  They give exact coordinate
translations; the slope and logarithm laws obtained from them use
classical joint-linear-complexity inputs.

\subheading{What is new and what is classical}
Here $M_D^{(r)}$ denotes the greatest common
cancellation depth attainable by a denominator of coefficient length
at most $D$.  Formal definitions are given in
Sections~\ref{sec:denominator-profile} and \ref{sec:joint}.
The classical slope from \cite{NiederreiterWang},
\[
        \JL_r(N)\sim\frac{r}{r+1}N
\]
is dual to $M_D^{(r)}\sim D/r$.
For $r=1$, Proposition~\ref{prop:one-series} identifies the renewal,
plateau, binomial, and fluctuation laws with classical continued-
fraction results
\cite{NiederreiterProbabilistic,LertchoosakulNairMetric,
LertchoosakulNairQuantitative,VielhaberIsometry}.

The new statement for $r\ge2$ is that the full residual vectors
revealed immediately after the stopping times $T_n$ are iid uniform
on $\F_q^r$.
Thus the jump indicators of the unique common minimal kernel line
are iid with parameter $1-q^{-r}$, and nonzero residuals carry
independent uniform projective marks.  The exact jump-size clock and
its uniform tail estimate further describe what happens after a
nonzero residual.  Multi-continued-fraction and generalized
Berlekamp--Massey algorithms reveal scalar discrepancies sequentially
and track the full symbol-by-symbol joint-complexity state
\cite{DaiFengYang,VielhaberBDM,VielhaberCanales}.  By contrast, the
present residual is an $r$-vector computed from one common monic
generator, all of whose coordinates use the same pre-update
denominator, and it is sampled only at the random horizons $T_n+1$.
Thus the Bernoulli vector law does not follow by multiplying the
sequential scalar discrepancy laws.  The common-denominator coupling
precludes a reduction to a product of one-series laws and yields the
simultaneous density in Theorem~\ref{thm:intro-density}.

\subheading{Relation with higher-dimensional L\'evy--Khintchine laws}
In the Archimedean setting, Cheung and Chevallier established a
L\'evy--Khintchine theorem for best simultaneous approximations, and
Aggarwal and Ghosh proved effective and central limit refinements
\cite{CheungChevallier,AggarwalGhosh}.  The present equal-depth
ultrametric problem for $r\ge2$ is different in both its
approximation profile and its probability structure.  The product
coefficient model makes the minimal-kernel persistence probability
exactly $q^{-r}$ and
leads to a closed formula for the spacing constant.

\subheading{Organization}
Sections~\ref{sec:local-hankel} and \ref{sec:denominator-profile} introduce
the stacked Hankel kernels and the classical slope input.
Section~\ref{sec:renewal} proves the exact marked renewal theorem and
the jump-size clock, followed by finite-time, central limit, and
iterated-logarithm consequences.
Section~\ref{sec:density} proves the density theorem for attained
minimal denominator lengths.
Section~\ref{sec:joint} gives exact dualities with joint linear
complexity and record slack.

\section{Laurent series and stacked Hankel kernels}
\label{sec:local-hankel}

For
\[
        x=\sum_{j\le N}a_jt^j\in K,
\]
write
\[
        x=\pp{x}+\fp{x},
        \qquad
        \pp{x}\in\F_q[t],
        \quad
        \fp{x}\in\pinf.
\]
For $x\neq0$, put
\[
        \deg x=\max\{j:a_j\neq0\},
        \qquad |x|=q^{\deg x},
\]
and set $|0|=0$.  The distance to the polynomial ring is
\[
        \norm{x}=|\fp{x}|.
\]

\subheading{Additional notation}
The coefficient-extraction notation $[t^m]x$ introduced above is
distinct from the polynomial-part notation $\pp{x}$.  The superscript
$\mathsf T$ denotes transpose.  For random variables,
$\stackrel{\mathrm d}=$ denotes equality in distribution and
$\Longrightarrow$ denotes convergence in distribution.  The symbols
$O(\,\cdot\,)$ and $o(\,\cdot\,)$ have their usual asymptotic meanings,
and $\log_q$ is the logarithm to base $q$.

\begin{definition}[Cancellation length]
If
\[
        \fp{x}=a_1t^{-1}+a_2t^{-2}+\cdots,
\]
define
\[
        c(x)=\max\{n\ge0:a_1=\cdots=a_n=0\}.
\]
If $\fp{x}=0$, set $c(x)=\infty$.
\end{definition}

Thus
\[
        c(x)\ge n
        \quad\Longleftrightarrow\quad
        \norm{x}\le q^{-(n+1)}.
\]
For $\boldsymbol\alpha=(\alpha_1,\ldots,\alpha_r)\in\pinf^r$ and
$0\neq Q\in\F_q[t]$, put
\[
        m_r(Q)=\min_{1\le i\le r}c(Q\alpha_i).
\]

Let $V_0=\{0\}$, and for $d\ge 1$, let
\(
        V_d=\{Q\in\F_q[t]:\deg Q<d\}
\).
The elements of 
\[
 \mathbb P(V_d)
 =
 (V_d\setminus\{0\})/\,\F_q^\times;
\]
are called \emph{projective denominator lines}.
Write
\[
        Q(t)=\theta_0+\cdots+\theta_{d-1}t^{d-1},
        \qquad
        \coef(Q)=(\theta_0,\ldots,\theta_{d-1})^{\mathsf T}.
\]

For
\[
        \alpha_i=\sum_{j\ge1}u_j^{(i)}t^{-j},
\]
the coefficient of $t^{-a}$ in $Q\alpha_i$ is
\[
        \sum_{k=0}^{d-1}\theta_ku_{a+k}^{(i)}.
\]
For $n\ge0$, let $H_{\boldsymbol\alpha}^{(r)}(n;d)$ be the
$rn\times d$ matrix obtained by stacking, for $1\le i\le r$, the
Hankel blocks
\[
 \begin{pmatrix}
 u_1^{(i)}&u_2^{(i)}&\cdots&u_d^{(i)}\\
 u_2^{(i)}&u_3^{(i)}&\cdots&u_{d+1}^{(i)}\\
 \vdots&\vdots&&\vdots\\
 u_n^{(i)}&u_{n+1}^{(i)}&\cdots&u_{n+d-1}^{(i)}
 \end{pmatrix}.
\]
For $n=0$ this is the empty matrix with $d$ columns.  Define
\[
        \cK_d^{(r)}(n)
        =
        \ker H_{\boldsymbol\alpha}^{(r)}(n;d)
        \subseteq V_d,
        \qquad
        \cK_0^{(r)}(n)=\{0\}.
\]

\begin{lemma}[Kernel criterion]\label{lem:kernel}
For $0\neq Q\in V_d$,
\[
        Q\in\cK_d^{(r)}(n)
        \quad\Longleftrightarrow\quad
        m_r(Q)\ge n.
\]
\end{lemma}

\begin{proof}
The matrix equation
$H_{\boldsymbol\alpha}^{(r)}(n;d)\coef(Q)=0$ says exactly that the
coefficients of $t^{-1},\ldots,t^{-n}$ in every $Q\alpha_i$ vanish.
\end{proof}

The kernels satisfy
\[
        \cK_d^{(r)}(n+1)\subseteq\cK_d^{(r)}(n),
        \qquad
        \cK_{d-1}^{(r)}(n)\subseteq\cK_d^{(r)}(n).
\]

\section{Minimal common-denominator profiles and the slope input}
\label{sec:denominator-profile}

\begin{definition}[Minimal denominator profile and best depth]
For $n\ge0$, define
\[
        L_r(n)
        =
        \min\{d\ge1:\cK_d^{(r)}(n)\neq0\}.
\]
For $D\ge1$, define
\[
        M_D^{(r)}
        =
        \max\{m_r(Q):0\neq Q\in V_D\}.
\]
The second quantity is allowed to be $\infty$ on the exceptional set
where one denominator gives exact cancellation in all coordinates.
\end{definition}

The function $L_r(n)$, called the minimal common-denominator length
profile, is nondecreasing and finite: the matrix
$H_{\boldsymbol\alpha}^{(r)}(n;rn+1)$ has more columns than rows.
For a fixed depth $n$, the feasible coefficient lengths are precisely
the integers $d\ge L_r(n)$.  Thus the staircase
\[
 \bigl\{(n,L_r(n)):n\in\Nzero\bigr\}
\]
is the lower feasibility boundary of the region where the stacked
Hankel kernel is nonzero.  Figures~\ref{fig:stopping-time-update} and
\ref{fig:profile-shear} depict this boundary schematically as a
frontier; here ``frontier'' is only a visual shorthand for the graph
of the profile.
Whenever $M_D^{(r)}<\infty$,
\[
        M_D^{(r)}
        =
        \max\{n:L_r(n)\le D\}.
\]

\begin{lemma}[The minimal kernel is a line]\label{lem:minimal-line}
For every $n\ge0$, if $D=L_r(n)$, then
\[
        \dim_{\F_q}\cK_D^{(r)}(n)=1.
\]
\end{lemma}

\begin{proof}
Minimality gives
\[
        \cK_{D-1}^{(r)}(n)=0,
        \qquad
        \cK_D^{(r)}(n)\neq0.
\]
If the latter kernel had dimension at least two, a nontrivial linear
combination of two independent elements could be chosen with zero
$t^{D-1}$ coefficient.  It would then be a nonzero element of
$\cK_{D-1}^{(r)}(n)$, a contradiction.
\end{proof}

\begin{definition}[Attained minimal denominator lengths]
The attained minimal denominator lengths are the distinct values of $L_r(n)$,
listed in increasing order.  The list may be finite on an exceptional
system and is denoted
\[
        d_1^{(r)}<d_2^{(r)}<\cdots
\]
whenever it is infinite.
\end{definition}

The attained minimal denominator lengths are equivalently the jump points of
$D\mapsto M_D^{(r)}$.  Indeed,
\[
        D=L_r(n)\text{ for some }n
        \quad\Longleftrightarrow\quad
        M_D^{(r)}>M_{D-1}^{(r)},
\]
with the convention $M_0^{(r)}=-\infty$.
To verify the equivalence, note that
$M_D^{(r)}\ge n$ exactly when $L_r(n)\le D$.  Thus
$L_r(n)=D$ implies
$M_{D-1}^{(r)}<n\le M_D^{(r)}$.  Conversely, if
$M_D^{(r)}>M_{D-1}^{(r)}$, any integer
$n$ with $M_{D-1}^{(r)}<n\le M_D^{(r)}$ satisfies $L_r(n)=D$; when
$M_D^{(r)}=\infty$, choose any $n>M_{D-1}^{(r)}$.

\begin{proposition}[Fixed-denominator cancellation]\label{prop:fixedQ}
For a fixed $0\neq Q\in\F_q[t]$ and Haar-random
$\boldsymbol\alpha\in\pinf^r$,
\[
        \Prob(m_r(Q)\ge n)=q^{-rn}.
\]
\end{proposition}

\begin{proof}
Write
\[
        Q(t)=\theta_0+\cdots+\theta_\ell t^\ell,
        \qquad \theta_\ell\neq0.
\]
For a fixed coordinate $i$, the $n$ cancellation equations are
triangular in the successive pivot variables
\[
        u_{\ell+1}^{(i)},\ldots,u_{\ell+n}^{(i)}.
\]
Their residual vector is therefore uniform on $\F_q^n$, and the
probability of its vanishing is $q^{-n}$.  The $r$ coefficient
sequences are independent, giving $q^{-rn}$.
\end{proof}

\begin{proposition}[Best-depth slope]\label{prop:slope}
For Haar-almost every $\boldsymbol\alpha\in\pinf^r$,
\[
        \lim_{D\to\infty}\frac{M_D^{(r)}}D=\frac1r.
\]
\end{proposition}

\begin{proof}
If $rn<D$, then $H_{\boldsymbol\alpha}^{(r)}(n;D)$ has fewer rows
than columns.  Hence
\[
        M_D^{(r)}
        \ge
        \left\lfloor\frac{D-1}{r}\right\rfloor.
\]

For the upper bound, the fixed-denominator law gives
\[
 \E\bigl(\#\cK_D^{(r)}(n)-1\bigr)
 =
 (q^D-1)q^{-rn}.
\]
For fixed $\varepsilon>0$, put
\[
        n_D=\left\lceil\left(\frac1r+\varepsilon\right)D\right\rceil.
\]
Then Markov's inequality gives
\[
 \Prob(M_D^{(r)}\ge n_D)
 \le
 \E\bigl(\#\cK_D^{(r)}(n_D)-1\bigr)
 \le q^{-r\varepsilon D}.
\]
The last inequality follows from
$rn_D\ge(1+r\varepsilon)D$.  Since the resulting geometric series is
summable in $D$, Borel--Cantelli gives
$M_D^{(r)}<n_D$ for all sufficiently large $D$, almost surely.
Intersecting these full-measure events over positive rational
$\varepsilon$ and combining with the deterministic lower bound proves
the stated limit.
\end{proof}

\begin{remark}
Proposition~\ref{prop:slope} is the inverse-coordinate form of the
classical typical profile slope
$\JL_r(N)/N\to r/(r+1)$ from joint linear complexity
\cite{NiederreiterWang}.
\end{remark}

\section{Exact renewal from fresh residuals at stopping times}
\label{sec:renewal}

Let
\[
        D_n=L_r(n).
\]
By Lemma~\ref{lem:minimal-line}, $\cK_{D_n}^{(r)}(n)$ has a unique
monic generator
\[
        Q_n(t)
        =
        t^{D_n-1}
        +\sum_{k=0}^{D_n-2}\theta_{n,k}t^k.
\]
Let
\[
 \mathcal F_N
 =
 \sigma\bigl(
 u_j^{(i)}:
 1\le i\le r,\ 1\le j\le N
 \bigr),
\]
so $(\mathcal F_N)_{N\ge0}$ is the coefficient filtration.  A
$\Nzero$-valued random variable $T$ is a stopping time if
$\{T\le N\}\in\mathcal F_N$ for every $N$.  Its stopped
sigma-algebra is
\[
 \mathcal F_T
 =
 \{A:A\cap\{T\le N\}\in\mathcal F_N
       \text{ for every }N\ge0\}.
\]
Now set
\[
        T_n=n+D_n-1.
\]
The random index $T_n$ is the earliest coefficient horizon at which a
common denominator of cancellation depth $n$ exists.  The next lemma
shows that $T_n$ is a stopping time and that the corresponding monic
minimal denominator $Q_n$ is $\mathcal F_{T_n}$-measurable.  Thus
$\boldsymbol u_{T_n+1}$ is the first fresh coefficient vector beyond
the stopped sigma-algebra.

\begin{figure}[!htbp]
\centering
\begin{tikzpicture}[
  x=1cm,
  y=0.72cm,
  >=Latex,
  every node/.style={font=\small},
  redlabel/.style={
    draw=freshred!55,
    text=freshred,
    fill=white,
    rounded corners=2pt,
    inner sep=3pt,
    align=left
  },
  greenlabel/.style={
    draw=survivegreen!55,
    text=survivegreen,
    fill=white,
    rounded corners=2pt,
    inner sep=3pt,
    align=left
  }
]
  \draw[->] (-0.15,0) -- (13.15,0)
    node[right] {depth};
  \draw[->] (0,-0.15) -- (0,6.25)
    node[above] {coefficient length};
  \node[below=2pt] at (5.7,0) {$n$};
  \node[below=2pt] at (7.25,0) {$n+1$};
  \node[left=3pt] at (0,2.6) {$D_n$};

  \draw[gray!55,densely dotted]
    (7.25,0.15) -- (7.25,4.35);

  \draw[profileblue,very thick]
    (0,0.85) -- (2.0,0.85) -- (2.0,1.65)
    -- (3.9,1.65) -- (3.9,2.6) -- (5.7,2.6);
  \coordinate (profilepoint) at (5.7,2.6);
  \fill[profileblue] (profilepoint) circle (2.2pt);
  \node[
    profileblue,
    fill=white,
    inner sep=1.5pt,
    below left=4pt and 5pt of profilepoint
  ]
    {$(n,D_n)$};

  \draw[survivegreen,very thick,dashed,->]
    (profilepoint) -- (12.65,2.6);
  \node[greenlabel] at (9.65,1.55)
    {$\rho_n^{(r)}=0$: the minimal kernel persists\\
     $J_n^{(r)}=0$};

  \draw[freshred,very thick,dashed,->]
    (profilepoint) -- (7.25,4.35) -- (12.65,4.35);
  \node[redlabel] at (10.05,5.35)
    {$\rho_n^{(r)}\ne0$: the profile jumps\\
     $J_n^{(r)}=1$};

  \node[redlabel,align=center] (fresh) at (3.0,5.05)
    {fresh coefficient vector $\boldsymbol u_{T_n+1}$\\
     $T_n=n+D_n-1$};
  \draw[freshred,->,thick]
    (fresh.south east) to[out=-12,in=145] (profilepoint);
\end{tikzpicture}
\caption{A schematic frontier picture of the one-step update after
the stopping time $T_n$.  The solid blue staircase represents the
minimal common-denominator length profile already determined through
coefficient index $T_n$; the fresh vector
at $T_n+1$ selects one of the two dashed continuations.  The diagram is
schematic: a nonzero residual may produce a jump larger than one in
coefficient length.}
\label{fig:stopping-time-update}
\end{figure}
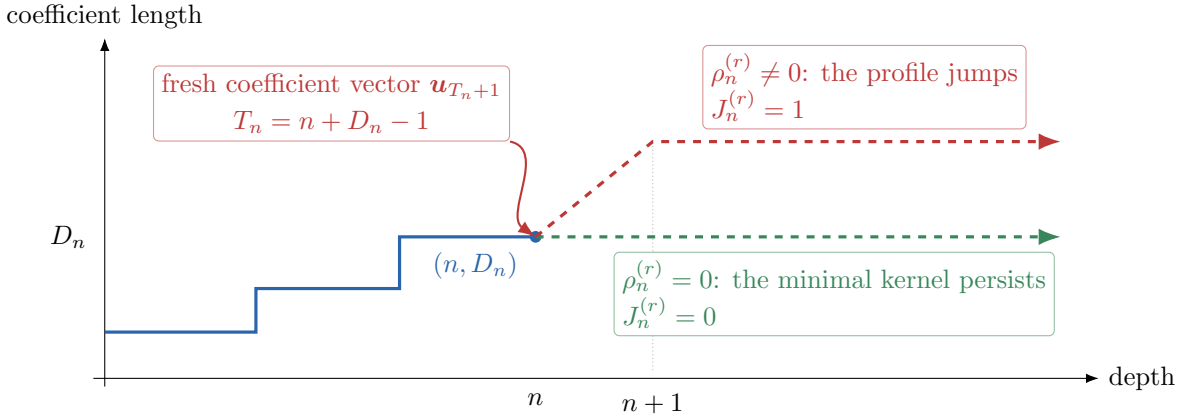
\FloatBarrier

\begin{lemma}[The stopping-time property]\label{lem:stopping-time-property}
For every fixed $n$, $T_n$ is a stopping time for
$(\mathcal F_N)_{N\ge0}$, and $Q_n$ is
$\mathcal F_{T_n}$-measurable.
\end{lemma}

\begin{proof}
For $N<n$, the event $\{T_n\le N\}$ is empty.  For $N\ge n$,
\[
 \{T_n\le N\}
 =
 \{L_r(n)\le N-n+1\}
 =
 \{\cK_{N-n+1}^{(r)}(n)\neq0\}.
\]
The last event is determined by coefficient vectors of indices at
most $N$, proving the stopping-time assertion.  On $\{T_n=N\}$, the
one-dimensional minimal kernel and its monic generator are likewise
determined by $\mathcal F_N$.  More explicitly, for every monic
polynomial $Q$ and every $N$,
\[
        \{Q_n=Q\}\cap\{T_n=N\}\in\mathcal F_N.
\]
Consequently,
\[
 \{Q_n=Q\}\cap\{T_n\le N\}
 =
 \bigcup_{m=0}^{N}
 \bigl(\{Q_n=Q\}\cap\{T_n=m\}\bigr)
 \in\mathcal F_N,
\]
which is precisely the $\mathcal F_{T_n}$-measurability of the
countably valued random variable $Q_n$.
\end{proof}

Define the next residual vector
\[
 \rho_n^{(r)}
 =
 \bigl(
 [t^{-(n+1)}]Q_n\alpha_1,\ldots,
 [t^{-(n+1)}]Q_n\alpha_r
 \bigr)
 \in\F_q^r.
\]

On the event $\{D_n=d\}$, the newly appended Hankel rows have the
following form; the entries in the last column constitute the
coefficient vector $\boldsymbol u_{T_n+1}=\boldsymbol u_{n+d}$.

\begin{figure}[!htbp]
\centering
\begin{tikzpicture}[
  >=Latex,
  every node/.style={font=\small}
]
  \matrix (newrows) [
    matrix of math nodes,
    nodes={
      draw=gray!65,
      minimum height=8mm,
      minimum width=22mm,
      text height=2.4ex,
      text depth=1.0ex,
      inner sep=1.5pt,
      outer sep=0pt,
      anchor=center
    },
    column sep=-\pgflinewidth,
    row sep=1.6mm,
    column 1/.style={nodes={fill=knownfill}},
    column 2/.style={nodes={fill=knownfill}},
    column 3/.style={nodes={fill=knownfill}},
    column 4/.style={nodes={fill=freshfill,draw=freshred}}
  ] {
    u_{n+1}^{(1)} & \cdots & u_{n+d-1}^{(1)}
      & u_{n+d}^{(1)}\\
    \vdots & & \vdots & \vdots\\
    u_{n+1}^{(r)} & \cdots & u_{n+d-1}^{(r)}
      & u_{n+d}^{(r)}\\
  };

  \node[above=2mm of newrows]
    {new rows added when the depth changes from $n$ to $n+1$};

  \draw[
    decorate,
    decoration={brace,mirror,amplitude=5pt}
  ] (newrows-3-1.south west) -- (newrows-3-3.south east)
    node[midway,below=6pt,align=center]
    {$\mathcal F_{T_n}$-measurable\\entries};
  \draw[
    freshred,
    decorate,
    decoration={brace,mirror,amplitude=5pt}
  ] (newrows-3-4.south west) -- (newrows-3-4.south east)
    node[midway,below=6pt,align=center,text=freshred]
    {fresh vector\\at $T_n+1$};

  \node[right=9mm of newrows] (times) {$\displaystyle
    \begin{pmatrix}
      \theta_{n,0}\\[-1mm]
      \vdots\\[-1mm]
      \theta_{n,d-2}\\
      1
    \end{pmatrix}$};
  \node[left=2mm of times] {$\times$};
  \node[right=8mm of times,align=center] (residual)
    {$=\ \rho_n^{(r)}$\\[-1mm]
     \scriptsize(conditionally uniform)};
\end{tikzpicture}
\caption{Only the last entry in each newly appended Hankel row is
outside the stopped sigma-algebra.  Monicity of $Q_n$ makes the
coefficient of this fresh vector equal to one.  When $d=1$, the known
block is empty.}
\label{fig:fresh-hankel-boundary}
\end{figure}
\FloatBarrier

\begin{lemma}[Fresh residual vector]\label{lem:fresh}
For every $a\in\F_q^r$,
\[
 \Prob\bigl(\rho_n^{(r)}=a\mid\mathcal F_{T_n}\bigr)=q^{-r}.
\]
The conditional identity is understood almost surely.
Thus, conditionally on $\mathcal F_{T_n}$, the residual vector is
uniform on $\F_q^r$.  In particular,
\[
 \Prob\bigl(\rho_n^{(r)}=0\mid\mathcal F_{T_n}\bigr)=q^{-r}.
\]
\end{lemma}

\begin{proof}
On $\{D_n=d\}$, monicity gives, for $1\le i\le r$,
\[
 [t^{-(n+1)}]Q_n\alpha_i
 =
 u_{n+d}^{(i)}
 +
 \sum_{k=0}^{d-2}\theta_{n,k}u_{n+1+k}^{(i)}.
\]
The sum is $\mathcal F_{T_n}$-measurable.  To justify the fresh-vector
claim at the random index, let $A\in\mathcal F_{T_n}$ and
$a\in\F_q^r$.  Since
$A\cap\{T_n=N\}\in\mathcal F_N$ and
$\boldsymbol u_{N+1}$ is independent of $\mathcal F_N$,
\[
\begin{aligned}
 \Prob\bigl(A\cap\{\boldsymbol u_{T_n+1}=a\}\bigr)
 &=
 \sum_{N\ge0}
 \Prob\bigl(A\cap\{T_n=N\}\cap\{\boldsymbol u_{N+1}=a\}\bigr)\\
 &=q^{-r}\sum_{N\ge0}\Prob\bigl(A\cap\{T_n=N\}\bigr)
 =q^{-r}\Prob(A).
\end{aligned}
\]
Thus $\boldsymbol u_{T_n+1}$ is independent of
$\mathcal F_{T_n}$ and uniform on $\F_q^r$.  Adding the measurable
sum preserves conditional uniformity and proves the lemma.
\end{proof}

\begin{theorem}[Exact marked renewal]\label{thm:renewal}
The residual vectors
\[
        \rho_0^{(r)},\rho_1^{(r)},\ldots
\]
are iid and uniformly distributed on $\F_q^r$.  Moreover,
\[
        J_n^{(r)}
        =
        \one_{\{L_r(n+1)>L_r(n)\}}
        =
        \one_{\{\rho_n^{(r)}\ne0\}},
        \qquad n\ge0,
\]
so the variables $J_n^{(r)}$ are iid Bernoulli with parameter
\[
        p_r=1-q^{-r}.
\]
\end{theorem}

\begin{proof}
Since
\[
        \cK_{D_n}^{(r)}(n)=\F_qQ_n,
\]
one has
\[
 \cK_{D_n}^{(r)}(n+1)
 =
 \begin{cases}
 \F_qQ_n,&\rho_n^{(r)}=0,\\
 0,&\rho_n^{(r)}\neq0.
 \end{cases}
\]
Because $L_r(n+1)\ge L_r(n)$,
\[
        J_n^{(r)}
        =
        \one_{\{\rho_n^{(r)}\neq0\}}.
\]

It remains to pass from conditional uniformity to independence.  If
$k<n$, then
$D_k\le D_n$ and
\[
        T_k+1=k+D_k\le n-1+D_n=T_n.
\]
Since $T_k$ is a stopping time, so is $T_k+1$, and
$\rho_k^{(r)}$ is $\mathcal F_{T_k+1}$-measurable.
The standard inclusion of stopped sigma-algebras,
\[
        S\le T\quad\Longrightarrow\quad\mathcal F_S\subseteq\mathcal F_T,
\]
therefore gives
\[
        \sigma(\rho_0^{(r)},\ldots,\rho_{n-1}^{(r)})
        \subseteq\mathcal F_{T_n}.
\]
Lemma~\ref{lem:fresh} now implies, for every $a\in\F_q^r$,
\[
 \Prob\bigl(\rho_n^{(r)}=a\mid
 \rho_0^{(r)},\ldots,\rho_{n-1}^{(r)}\bigr)
 =q^{-r}.
\]
Iteration proves that the residual vectors are iid uniform.  The
Bernoulli assertion follows by applying the nonzero indicator to each
residual vector.
\end{proof}

\begin{corollary}[Plateaux and projective marks]\label{cor:plateaux}
Set $S_0=0$ and, recursively,
\[
 S_j=\inf\{N>S_{j-1}:J_{N-1}^{(r)}=1\},
 \qquad
 H_j=S_j-S_{j-1}.
\]
Here $\inf\varnothing=\infty$; Theorem~\ref{thm:renewal} implies that
all these times are finite almost surely.
The times $S_1,S_2,\ldots$ are the renewal epochs.  The profile
$L_r(n)$ is constant for
$S_{j-1}\le n\le S_j-1$, so $H_j$ is the length of the $j$th
plateau as well as the $j$th inter-renewal time.  The variables
$H_1,H_2,\ldots$ are iid with
\[
        \Prob(H_j=h)
        =
        (1-q^{-r})q^{-r(h-1)},
        \qquad h\ge1.
\]
In particular,
\[
        \E H_j=\frac1{1-q^{-r}}.
\]
If
\[
        \Xi_j=[\rho_{S_j-1}^{(r)}]
        \in\mathbb P^{r-1}(\F_q),
\]
then $\Xi_1,\Xi_2,\ldots$ are iid uniform on
$\mathbb P^{r-1}(\F_q)$ and are independent of the plateau lengths
$(H_j)_{j\ge1}$.  For $r=1$, the projective space has one element.
\end{corollary}

\begin{proof}
The times $S_j$ are the renewal epochs of the iid Bernoulli sequence
$(J_n^{(r)})_{n\ge0}$.  Thus the increments $H_j$ are iid
geometric waiting times with success parameter $1-q^{-r}$.
Conditionally on a residual vector being nonzero, it is uniform on
$\F_q^r\setminus\{0\}$.  Every projective line contains exactly
$q-1$ nonzero vectors.  Hence, for every
$\xi\in\mathbb P^{r-1}(\F_q)$,
\[
 \Prob\bigl([\rho_n^{(r)}]=\xi\mid\rho_n^{(r)}\ne0\bigr)
 =
 \frac{q-1}{q^r-1}
 =
 \frac1{\#\mathbb P^{r-1}(\F_q)}.
\]
This conditional distribution does not depend on the zero--nonzero
pattern.  Applying this factorization independently to the iid
residual vectors shows that the successive projective marks are iid
and independent of the waiting times.
\end{proof}

Put
\[
        \Delta_n^{(r)}=D_{n+1}-D_n,
        \qquad
        \mathcal S_n^{(r)}=rn+1-D_n.
\]
The nonnegative integer $\mathcal S_n^{(r)}$ measures how far the
minimal length lies below the dimension bound $rn+1$.  We next refine
the event $\{\Delta_n^{(r)}>0\}$ by following kernel dimensions while
the coefficient length, rather than the depth, is increased.

On $\{D_n=d\}$, define
\[
 B_{n,k}=H_{\boldsymbol\alpha}^{(r)}(n;d+k),
 \qquad
 \nu_{n,k}=\dim\ker B_{n,k}
 \quad(k\ge0).
\]
Lemma~\ref{lem:minimal-line} gives $\nu_{n,0}=1$.  Appending one
column changes the nullity by either zero or one, so for
$1\le j\le r$ the kernel-growth epochs
\[
 \tau_{n,j}
 =
 \inf\{k\ge1:\nu_{n,k}=j+1\}
\]
are well defined and satisfy
\begin{equation}
 \tau_{n,j}\le rn+j+1-d.
 \label{eq:kernel-growth-bound}
\end{equation}
Indeed, the matrix $B_{n,k}$ has $rn$ rows, and hence has nullity at
least $d+k-rn$.

For fixed $n$, put
\[
        \mathcal G_{n,k}=\mathcal F_{T_n+k},
        \qquad k\ge0.
\]
Because $T_n+k$ is a stopping time and these random times increase
with $k$, $(\mathcal G_{n,k})_{k\ge0}$ is a filtration.  For a
stopping time $\sigma$ relative to this filtration, write
$\mathcal G_{n,\sigma}$ for its stopped sigma-algebra.

\begin{auxlemma}[Freshness at shifted stopping times]
Let $\sigma$ be an almost surely finite stopping time for
$(\mathcal G_{n,k})_{k\ge0}$.  Then, for every $a\in\F_q^r$,
\[
 \Prob\bigl(
 \boldsymbol u_{T_n+\sigma+1}=a
 \mid \mathcal G_{n,\sigma}
 \bigr)
 =q^{-r}
 \qquad\text{almost surely}.
\]
\end{auxlemma}

\begin{proof}
The random variable $R=T_n+\sigma$ is a stopping time for the original
coefficient filtration.  Indeed, this follows from the stopping-time
property of $\sigma$ after decomposing according to the values of
$T_n$: if $A\in\mathcal F_{T_n+k}$, then
$A\cap\{T_n=m\}\in\mathcal F_{m+k}$.  Consequently, for every $N$,
\[
 \{R\le N\}
 =
 \bigcup_{m=0}^{N}
 \bigl(
 \{T_n=m\}\cap\{\sigma\le N-m\}
 \bigr)\in\mathcal F_N.
\]
The definitions also give
$\mathcal F_R=\mathcal G_{n,\sigma}$.  If
$A\in\mathcal F_R$, independence of
$\boldsymbol u_{N+1}$ from $\mathcal F_N$ gives
\[
\begin{aligned}
 \Prob\bigl(
 A\cap\{\boldsymbol u_{R+1}=a\}
 \bigr)
 &=
 \sum_{N\ge0}
 \Prob\bigl(
 A\cap\{R=N\}\cap\{\boldsymbol u_{N+1}=a\}
 \bigr)\\
 &=q^{-r}\sum_{N\ge0}\Prob(A\cap\{R=N\})
 =q^{-r}\Prob(A).
\end{aligned}
\]
This proves the conditional identity.
\end{proof}

For $1\le j\le r$, the epoch $\tau_{n,j}$ is an almost surely finite
stopping time for $(\mathcal G_{n,k})_{k\ge0}$, since
\[
 \{\tau_{n,j}\le k\}
 =
 \{\nu_{n,k}\ge j+1\}\in\mathcal G_{n,k},
\]
and finiteness follows from \eqref{eq:kernel-growth-bound}.

\begin{theorem}[Positive-jump clock and tail]
\label{thm:jump-clock}
Define $I_n=0$ on $\{J_n^{(r)}=0\}$.  On the event
$\{J_n^{(r)}=1\}$, there is a unique random index
$I_n\in\{1,\ldots,r\}$ such that
\[
        \Delta_n^{(r)}=\tau_{n,I_n}.
\]
For $1\le j\le r$, the following global conditional
identity holds almost surely:
\begin{equation}
 \Prob(I_n=j\mid\mathcal F_{T_n+1})
 =
 \one_{\{J_n^{(r)}=1\}}\,
 q^{j-r}\prod_{a=1}^{j-1}(1-q^{a-r}),
 \qquad 1\le j\le r,
 \label{eq:jump-clock-law}
\end{equation}
where an empty product equals one.  Equivalently, after
restricting to $\{J_n^{(r)}=1\}$, these are the conditional
probabilities of the $r$ possible kernel-growth epochs.  Also,
\[
 \Prob(I_n=0\mid\mathcal F_{T_n+1})
 =
 \one_{\{J_n^{(r)}=0\}}.
\]
 Moreover,
\begin{equation}
        \Delta_n^{(r)}\le \mathcal S_n^{(r)}+r
        \qquad\text{on }\{J_n^{(r)}=1\},
 \label{eq:jump-deterministic-bound}
\end{equation}
and for every integer $h\ge r+1$,
\begin{equation}
 \Prob\bigl(\Delta_n^{(r)}\ge h\mid J_n^{(r)}=1\bigr)
 \le
 \frac{q^{r+1-h}}{q-1}.
 \label{eq:jump-tail}
\end{equation}
The estimate is uniform in $n$.
\end{theorem}

\begin{proof}
Work on $\{D_n=d,J_n^{(r)}=1\}$.  For $k\ge0$, set
\[
        A_k=H_{\boldsymbol\alpha}^{(r)}(n+1;d+k).
\]
The jump event says precisely that $A_0$ has full column rank.  Suppose
inductively that $A_{k-1}$ has full column rank, and let $W_{k-1}$ be
its column space.  Passing from $A_{k-1}$ to $A_k$ appends one column.
All but its last entry in each of the $r$ Hankel blocks are
$\mathcal F_{T_n+k}$-measurable, while the last $r$ entries form the
fresh uniform vector $\boldsymbol u_{T_n+k+1}$.

Let $E\subset\F_q^{r(n+1)}$ be the $r$-dimensional coordinate space
supported on those last $r$ entries.  The new column can belong to
$W_{k-1}$ only if its projection to the other $rn$ coordinates lies
in the projection of $W_{k-1}$.  This occurs exactly when the new
column of $B_{n,k}$ is in the column space of $B_{n,k-1}$, or,
equivalently,
\[
        \nu_{n,k}=\nu_{n,k-1}+1.
\]
In that case, because $A_{k-1}$ has full column rank, the
coefficient-to-column isomorphism $x\mapsto A_{k-1}x$ identifies
\[
 \ker B_{n,k-1}
 \quad\text{with}\quad
 E\cap W_{k-1}.
\]
Thus the fresh vector values that make the new column dependent form
an affine subspace of $\F_q^r$ of dimension $\nu_{n,k-1}$, and
\begin{equation}
 \Prob\bigl(A_k\text{ loses full column rank}
       \mid\mathcal F_{T_n+k},\ A_{k-1}\text{ full rank}\bigr)
 =q^{\nu_{n,k-1}-r}.
 \label{eq:column-hazard}
\end{equation}
If the nullity does not increase, this conditional probability is
zero.

Because $D_{n+1}=d+\Delta_n^{(r)}$ is the least coefficient length
for which $H_{\boldsymbol\alpha}^{(r)}(n+1;\,\cdot\,)$ has a nonzero
kernel, $\Delta_n^{(r)}$ is the least $k\ge1$ for which $A_k$ fails to
have full column rank.

We now justify the passage from the fixed-$k$ calculation to the
random growth epochs.  Put
\[
 h_j=q^{j-r},
 \qquad 1\le j\le r,
\]
and let $E_{n,j}$ be the event that the update has survived through
the first $j-1$ kernel-growth epochs.  Thus
\[
 E_{n,1}=\{J_n^{(r)}=1\},
 \qquad
 E_{n,j}
 =
 \{J_n^{(r)}=1,\ \Delta_n^{(r)}>\tau_{n,j-1}\}
 \quad(2\le j\le r).
\]
On $E_{n,j}$, the matrix $A_{\tau_{n,j}-1}$ has full column rank and
\[
        \nu_{n,\tau_{n,j}-1}=j.
\]
The event $E_{n,j}$ belongs to
$\mathcal G_{n,\tau_{n,j}}$.  The shifted-freshness lemma above,
applied with $\sigma=\tau_{n,j}$, and the affine-subspace calculation
in \eqref{eq:column-hazard} therefore give
\begin{equation}
 \Prob\bigl(
 \Delta_n^{(r)}=\tau_{n,j}
 \mid \mathcal G_{n,\tau_{n,j}}
 \bigr)
 =
 \one_{E_{n,j}}h_j.
 \label{eq:random-epoch-hazard}
\end{equation}
Equivalently, on $E_{n,j}$ the conditional failure probability at the
$j$th growth epoch is $h_j$, and the conditional survival probability
is $1-h_j$.

Since $\tau_{n,j}\ge1$, one has
$\mathcal F_{T_n+1}=\mathcal G_{n,1}
\subseteq\mathcal G_{n,\tau_{n,j}}$.  The tower property applied to
\eqref{eq:random-epoch-hazard} yields, successively,
\[
 \Prob(E_{n,j}\mid\mathcal F_{T_n+1})
 =
 \one_{\{J_n^{(r)}=1\}}
 \prod_{a=1}^{j-1}(1-h_a)
\]
and hence
\[
 \Prob(I_n=j\mid\mathcal F_{T_n+1})
 =
 \one_{\{J_n^{(r)}=1\}}\,
 h_j\prod_{a=1}^{j-1}(1-h_a).
\]
This is \eqref{eq:jump-clock-law}.  Finally, $h_r=1$, and therefore
\[
 \sum_{j=1}^{r}
 h_j\prod_{a=1}^{j-1}(1-h_a)
 =
 1-\prod_{a=1}^{r}(1-h_a)
 =1.
\]
Thus, on the jump event, failure occurs at exactly one of the listed
growth epochs.  The argument is independent of the value $d$; summing
over the $\mathcal F_{T_n}$-measurable partition
$\{D_n=d\}$ gives the asserted global conditional identity.
The bound \eqref{eq:kernel-growth-bound} with $j=r$
gives
\[
 \Delta_n^{(r)}\le\tau_{n,r}
 \le rn+r+1-D_n=\mathcal S_n^{(r)}+r,
\]
which is \eqref{eq:jump-deterministic-bound}.

It remains to control $\mathcal S_n^{(r)}$.  If $s\ge1$ and
$m=rn+1-s\ge1$, then $\{\mathcal S_n^{(r)}\ge s\}$ implies that at
least one projective denominator line in $\mathbb P(V_m)$ cancels to
depth $n$.  Proposition~\ref{prop:fixedQ} and the union bound give
\begin{equation}
 \Prob(\mathcal S_n^{(r)}\ge s)
 \le
 \frac{q^m-1}{q-1}\,q^{-rn}
 <\frac{q^{1-s}}{q-1}.
 \label{eq:slack-tail}
\end{equation}
When $m<1$ the event is empty, so the simpler final bound remains
valid.  Since $\mathcal S_n^{(r)}$ is
$\mathcal F_{T_n}$-measurable and Lemma~\ref{lem:fresh} makes
$J_n^{(r)}$ independent of $\mathcal F_{T_n}$, the same estimate
holds conditionally on $J_n^{(r)}=1$.  Combining
\eqref{eq:jump-deterministic-bound} with \eqref{eq:slack-tail}, using
$s=h-r$, proves \eqref{eq:jump-tail}.
\end{proof}

\begin{corollary}[The exact two-series clock]
\label{cor:r-two-clock}
Let $r=2$. With the convention $I_n=0$ on
$\{J_n^{(2)}=0\}$,
\[
 \Prob\bigl(I_n=1\mid\mathcal F_{T_n+1}\bigr)
 =\frac{\one_{\{J_n^{(2)}=1\}}}{q},
 \qquad
 \Prob\bigl(I_n=2\mid\mathcal F_{T_n+1}\bigr)
 =\one_{\{J_n^{(2)}=1\}}\left(1-\frac1q\right).
\]
After restricting to $\{J_n^{(2)}=1\}$, this says that
the first increase of $\dim\cK_{D_n+k}^{(2)}(n)$ decides the jump
with probability $1/q$; if it does not, the jump occurs surely at the
second increase.  In particular,
\[
 \Delta_n^{(2)}\le 2n+3-D_n
\]
and, for $h\ge3$,
\[
 \Prob\bigl(\Delta_n^{(2)}\ge h\mid J_n^{(2)}=1\bigr)
 \le\frac{q^{3-h}}{q-1}.
\]
\end{corollary}

\begin{example}[A short marked profile over $\F_2$]
\label{ex:marked-profile}
Take $r=2$, $q=2$, and begin the coefficient sequence with
\[
\begin{array}{c|cccccc}
j&1&2&3&4&5&6\\ \hline
\boldsymbol u_j&(0,0)&(1,0)&(1,0)&(0,0)&(1,0)&(1,1).
\end{array}
\]
The first four profile states are
\[
\begin{array}{c|c|c|c|c}
n&D_n&Q_n&\rho_n^{(2)}&J_n^{(2)}\\ \hline
0&1&1&(0,0)&0\\
1&1&1&(1,0)&1\\
2&3&1+t+t^2&(0,0)&0\\
3&3&1+t+t^2&(0,1)&1.
\end{array}
\]
Indeed, no monic polynomial $t+a$ reaches depth two, since its
$t^{-1}$ coefficient vector is
$\boldsymbol u_2+a\boldsymbol u_1=(1,0)$.  The only nonzero constant
polynomial is $1$, and it also fails at depth two because
$\boldsymbol u_2\ne0$.  On the other hand,
$Q=1+t+t^2$ does reach depth two because
\[
 \boldsymbol u_1+\boldsymbol u_2+\boldsymbol u_3=0,
 \qquad
 \boldsymbol u_2+\boldsymbol u_3+\boldsymbol u_4=0.
\]
Its next two residuals are
\[
 \boldsymbol u_3+\boldsymbol u_4+\boldsymbol u_5=(0,0),
 \qquad
 \boldsymbol u_4+\boldsymbol u_5+\boldsymbol u_6=(0,1).
\]
Thus the first residual extends the same plateau, whereas the second
ends it and carries the projective mark $[0:1]\in\mathbb P^1(\F_2)$.
This finite computation is illustrative only.  For a Haar-random
input, before conditioning on this particular prefix,
Theorem~\ref{thm:renewal} asserts that the analogous residual sequence
is iid.
\end{example}

The exact renewal theorem also supplies finite-time and fluctuation
statements without additional Diophantine input.  Put
\[
        C_N^{(r)}
        =
        1+\sum_{n=0}^{N-1}J_n^{(r)}.
\]
This is the number of distinct values among
$L_r(0),L_r(1),\ldots,L_r(N)$.

\begin{corollary}[Exact finite-time law and fluctuations]
\label{cor:fluctuations}
Let $p_r=1-q^{-r}$.  Then
\[
        C_N^{(r)}-1\stackrel{\mathrm d}=
        \operatorname{Bin}(N,p_r).
\]
Here $\operatorname{Bin}(N,p_r)$ denotes the binomial distribution
with $N$ trials and success probability $p_r$.
Consequently,
\[
        \frac{C_N^{(r)}}N\longrightarrow p_r
        \qquad\text{almost surely},
\]
and
\[
 \frac{C_N^{(r)}-1-Np_r}
 {\sqrt{Np_r(1-p_r)}}
 \ \Longrightarrow\ \mathcal N(0,1).
\]
Here $\mathcal N(0,1)$ is the standard normal distribution.
Moreover, almost surely,
\[
 \limsup_{N\to\infty}
 \frac{C_N^{(r)}-1-Np_r}
 {\sqrt{2Np_r(1-p_r)\log\log N}}
 =1,
\]
and the corresponding liminf equals $-1$.
\end{corollary}

\begin{proof}
The binomial identity follows directly from
Theorem~\ref{thm:renewal}.  The remaining statements are the strong
law, central limit theorem, and law of the iterated logarithm for iid
Bernoulli variables.
\end{proof}

\section{Density of attained minimal denominator lengths}
\label{sec:density}

Let
\[
        \cA_r(D)=\#\{j:d_j^{(r)}\le D\}.
\]
The best-depth slope implies that the list of attained lengths is infinite almost
surely.  We work below on this full-measure set.
Whenever $M_D^{(r)}<\infty$, the attained lengths at most $D$ are exactly
the distinct values of
\[
        L_r(0),L_r(1),\ldots,L_r(M_D^{(r)}).
\]
Therefore
\[
        \cA_r(D)=C_{M_D^{(r)}}^{(r)}.
\]
Proposition~\ref{prop:fixedQ} and a countable union over nonzero
polynomials show that $M_D^{(r)}<\infty$ for every $D$ almost surely:
for each fixed $Q\neq0$,
\[
 \Prob(m_r(Q)=\infty)
 =
 \lim_{n\to\infty}q^{-rn}=0.
\]

\begin{theorem}[Density and spacing of attained lengths]\label{thm:density}
For Haar-almost every $\boldsymbol\alpha\in\pinf^r$,
\[
        \lim_{D\to\infty}\frac{\cA_r(D)}D
        =
        \frac{1-q^{-r}}r.
\]
Equivalently,
\[
        \lim_{j\to\infty}\frac{d_j^{(r)}}j
        =
        \frac{r}{1-q^{-r}}.
\]
\end{theorem}

\begin{proof}
By Corollary~\ref{cor:fluctuations},
\[
        \frac{C_N^{(r)}}N\longrightarrow1-q^{-r}
        \qquad\text{almost surely}.
\]
By Proposition~\ref{prop:slope},
\[
        \frac{M_D^{(r)}}D\longrightarrow\frac1r.
\]
In particular, $M_D^{(r)}\to\infty$, so the almost-sure limit for
$C_N^{(r)}/N$ may be evaluated along the random integer sequence
$N=M_D^{(r)}$.
Hence
\[
 \frac{\cA_r(D)}D
 =
 \frac{C_{M_D^{(r)}}^{(r)}}{M_D^{(r)}}
 \frac{M_D^{(r)}}D
 \longrightarrow
 \frac{1-q^{-r}}r.
\]
Writing $\lambda=(1-q^{-r})/r$, one has
$\cA_r(d_j^{(r)})=j$.  Hence
\[
        \frac{j}{d_j^{(r)}}
        =
        \frac{\cA_r(d_j^{(r)})}{d_j^{(r)}}
        \longrightarrow\lambda,
\]
which proves the inverse-density, or spacing, statement.
\end{proof}

\begin{remark}
When $r=1$, Proposition~\ref{prop:one-series} gives
\[
        d_j^{(1)}
        =
        1+\sum_{k=1}^{j-1}\deg A_k.
\]
Thus the spacing statement in Theorem~\ref{thm:density} is also an
immediate form of the classical law
\[
        \frac1j\sum_{k=1}^j\deg A_k
        \longrightarrow\frac{q}{q-1}
\]
of Houndonougbo and Berth\'e--Nakada
\cite{Houndonougbo,BertheNakada,LertchoosakulNairQuantitative}.
The content newly asserted by Theorem~\ref{thm:density} is its
simultaneous $r\ge2$ specialization obtained from the exact renewal
at the stopping times and the best-depth slope.
\end{remark}

\begin{remark}
Corollary~\ref{cor:fluctuations} concerns the depth-indexed count
$C_N^{(r)}$.  A central limit theorem for the length-indexed count
$\cA_r(D)=C_{M_D^{(r)}}^{(r)}$ additionally requires sufficiently
precise fluctuation control for the random time $M_D^{(r)}$; it does
not follow from the Bernoulli theorem alone.
\end{remark}

\section{Exact duality with joint linear complexity}
\label{sec:joint}

The coefficient vectors
\[
        (u_j^{(1)},\ldots,u_j^{(r)})\in\F_q^r
\]
form an $r$-fold multisequence.  For $N\in\Nzero$, let $\JL_r(N)$
denote the least integer $0\le\ell\le N$ for which there exist
\[
        a_0,\ldots,a_\ell\in\F_q,
        \qquad a_\ell\neq0,
\]
satisfying
\[
        \sum_{k=0}^{\ell}a_ku_{j+k}^{(i)}=0
\]
for every $1\le i\le r$ and every
$1\le j\le N-\ell$.  Thus $\JL_r(0)=0$.  When $N\ge1$, the case
$\ell=0$ occurs exactly when all first $N$ coefficient vectors are
zero; this fixes explicitly the zero-prefix convention.
For $r=1$, the equivalence between recurrence length,
continued-fraction convergents, and linear complexity is classical
\cite{NiederreiterProbabilistic,NiederreiterCombinatorial,
VielhaberIsometry}.  The multisequence formulation and its typical
slope are treated in
\cite{NiederreiterJointSurvey,NiederreiterWang,DaiFengYang,
VielhaberCanales}.

A recurrence polynomial of degree $\ell$ for the first $N$ vectors
belongs to
\[
        \cK_{\ell+1}^{(r)}(N-\ell),
\]
whereas every nonzero element of this kernel supplies a recurrence
polynomial of degree at most $\ell$.  Putting
$d=\ell+1$ therefore gives the exact indexing relation
\[
 \JL_r(N)
 =
 \min_{\substack{1\le d\le N+1\\
                  \cK_d^{(r)}(N-d+1)\neq0}}
 (d-1).
\]

\begin{proposition}[Profile correspondence]
\label{prop:profile-correspondence}
For every $n\ge0$, put
\[
        d=L_r(n),
        \qquad
        N=n+d-1.
\]
Then
\[
        \JL_r(N)=d-1.
\]
\end{proposition}

\begin{proof}
The inclusion
$\cK_d^{(r)}(n)\neq0$ gives $\JL_r(N)\le d-1$.  If a recurrence of
degree at most $d-2$ existed for the first $N$ coefficient vectors,
its denominator would have coefficient length at most $d-1$ and
cancellation depth at least
\[
        N-(d-2)=n+1.
\]
It would in particular lie nontrivially in
$\cK_{d-1}^{(r)}(n)$, contradicting $d=L_r(n)$.
\end{proof}

Thus $L_r$ is a depth-parametrized counterpart of the joint
linear-complexity profile, rather than the joint linear-complexity
profile itself.  The exact change of coordinates is displayed below.

\begin{figure}[!htbp]
\centering
\begin{tikzpicture}[
  >=Latex,
  every node/.style={font=\small}
]
  \begin{scope}[x=0.54cm,y=0.52cm]
    \draw[->] (0,0) -- (7.8,0) node[right] {$n$};
    \draw[->] (0,0) -- (0,5.4) node[above] {$d$};
    \draw[profileblue,very thick]
      (0,1)--(2,1)--(2,2)--(4,2)--(4,3)--(7,3);
    \fill[profileblue] (4,3) circle (2.2pt);
    \node[above left] at (4,3) {$(n,d)$};
    \node[below=5mm] at (3.8,0) {$d=L_r(n)$};
  \end{scope}

  \node[
    align=center
  ] (shear) at (7.2,1.65)
    {$\Phi(n,d)=(N,\ell)$\\
     $N=n+d-1,\quad\ell=d-1$};
  \draw[->,thick] (4.3,1.65) -- (shear.west);
  \draw[->,thick] (shear.east) -- (10,1.65);

  \begin{scope}[xshift=10.2cm,x=0.54cm,y=0.52cm]
    \draw[->] (0,0) -- (8.2,0) node[right] {$N$};
    \draw[->] (0,0) -- (0,5.4) node[above] {$\ell$};
    \draw[survivegreen,very thick]
      (0,0)--(2,0)--(2,1)--(5,1)--(5,2)--(6,2)--(6,3)--(8,3);
    \fill[survivegreen] (6,2) circle (2.2pt);
    \node[above left] at (6,2) {$(N,\ell)$};
    \node[below=5mm] at (4.1,0) {$\ell=\JL_r(N)$};
  \end{scope}
\end{tikzpicture}
\caption{A schematic frontier picture of the shear relating a point
of the minimal common-denominator length profile to a selected point
of the joint-linear-complexity profile.  The displayed coordinate
change is exact.}
\label{fig:profile-shear}
\end{figure}
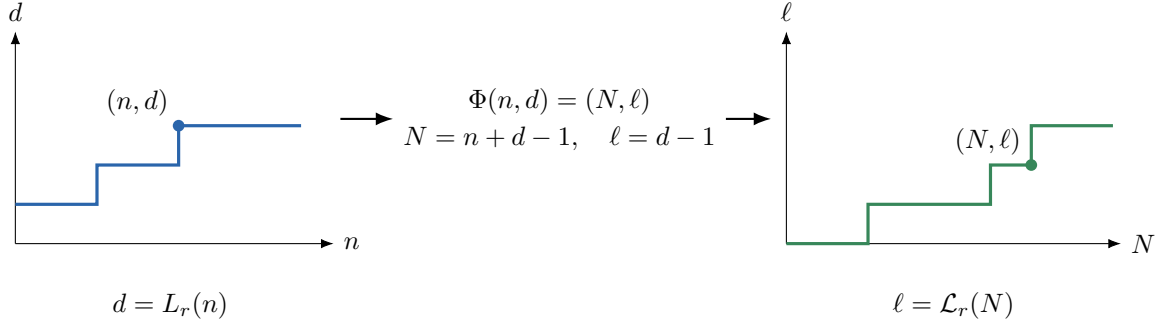
\FloatBarrier

Here is the precise slope consequence.  Put $d_n=L_r(n)$ and
$N_n=n+d_n-1$.  On the full-measure set where the classical profile
slope holds, Proposition~\ref{prop:profile-correspondence} gives
\[
 \frac{d_n-1}{n+d_n-1}
 =
 \frac{\JL_r(N_n)}{N_n}
 \longrightarrow\frac{r}{r+1}.
\]
For $n\ge1$, setting $x_n=(d_n-1)/n$ rewrites the left-hand side as
$x_n/(1+x_n)$.  Solving this relation shows that $x_n\to r$, and
therefore
\[
        \frac{L_r(n)}n\longrightarrow r.
\]
Finally, $M_D^{(r)}=\max\{n:L_r(n)\le D\}$ is the generalized inverse
of $L_r$.  Applying the inequalities
\[
 (r-\varepsilon)n\le L_r(n)\le(r+\varepsilon)n
\]
for all sufficiently large $n$, and then letting
$\varepsilon\downarrow0$, yields
$M_D^{(r)}/D\to1/r$.  This makes explicit the direction in which the
classical profile slope supplies the best-depth slope used above.

\subsection{Record slack is a joint-complexity deviation}

For $0\neq Q\in\F_q[t]$, define
\[
        \tau_r(Q)=rm_r(Q)-d(Q),
\]
and
\[
        R_D^{(r)}
        =
        \max\{\tau_r(Q):0\neq Q\in V_D\}.
\]
Here and below $D\ge1$.

\begin{proposition}[Exact record duality]\label{prop:record-duality}
Assume that no nonzero polynomial has infinite common cancellation
depth.  Then, for every $D\ge1$,
\[
 R_D^{(r)}
 =
 \max_{\substack{N\in\Nzero\\ \JL_r(N)+1\le D}}
 \bigl(rN-(r+1)\JL_r(N)-1\bigr).
\]
\end{proposition}

\begin{proof}
The maximum on the right is well defined.  Indeed, the eligible set
contains $N=0$.  If it were unbounded, then for arbitrarily large
$N$ there would be a nonzero recurrence denominator in the finite
set $V_D$.  One such denominator would occur for unboundedly many
$N$ and hence would have infinite common cancellation depth,
contrary to the hypothesis.

Let $Q$ have coefficient length $d\le D$ and depth $n=m_r(Q)$.
Set
\[
        N=n+d-1.
\]
Then $Q$ supplies a recurrence of degree $d-1$ for the first $N$
coefficient vectors, so $\JL_r(N)\le d-1$.  Therefore
\[
\begin{aligned}
 rN-(r+1)\JL_r(N)-1
 &\ge r(n+d-1)-(r+1)(d-1)-1\\
 &=rn-d\\
 &=\tau_r(Q).
\end{aligned}
\]
Also $\JL_r(N)+1\le d\le D$.  Taking the maximum over $Q$ proves one
inequality.

Conversely, let $\ell=\JL_r(N)$ with $\ell+1\le D$.  A minimal
recurrence denominator has coefficient length $\ell+1$ and
cancellation depth at least $N-\ell$.  Its slack is at least
\[
        r(N-\ell)-(\ell+1)
        =
        rN-(r+1)\ell-1.
\]
Taking the maximum over $N$ proves the reverse inequality.
\end{proof}

\begin{remark}
Proposition~\ref{prop:record-duality} is an exact deterministic
coordinate identity.  No logarithm law for $R_D^{(r)}$ is asserted
here; obtaining a self-contained sharp record law would require
additional estimates for joint-linear-complexity deviations.
\end{remark}

\section{Further questions}

The exact renewal theorem leaves several concrete problems that are
not settled by the present paper.

\subheading{Limiting positive-jump laws}
Theorem~\ref{thm:jump-clock} gives a depth-uniform tail bound and an
exact law for the kernel-growth index $I_n$.  It does not, however,
determine the distribution of the random epochs $\tau_{n,j}$, and
hence does not give a closed formula for the scalar jump size when
$r\ge2$.  For $r=1$, Proposition~\ref{prop:one-series} identifies the
jump sizes with partial-quotient degrees.  For $r\ge2$, do the
conditional laws
\[
 \operatorname{Law}\bigl(\Delta_n^{(r)}\mid J_n^{(r)}=1\bigr)
\]
converge as $n\to\infty$?  If so, one would like an explicit
probability generating function and the joint limiting law of
\[
        \bigl(\Delta_n^{(r)},[\rho_n^{(r)}]\bigr)
        \quad\text{conditioned on }J_n^{(r)}=1.
\]
A complementary finite-time problem is to identify a minimal state
that generates the successive epochs $\tau_{n,j}$; for $r=2$,
Corollary~\ref{cor:r-two-clock} reduces the issue to the joint law of
two such epochs.  An explicit transition kernel could yield
fluctuation results for the length-indexed count $\cA_r(D)$.  The jump
heights arising in
symbol-by-symbol joint-linear-complexity algorithms are related but
are not the depth-coordinate increments $\Delta_n^{(r)}$; the shear in
Proposition~\ref{prop:profile-correspondence} is essential when
comparing the two observables
\cite{DaiFengYang,VielhaberBDM,VielhaberCanales}.

\subheading{Projectivized extreme-value processes}
Recall that
$\mathbb P(V_D)=(V_D\setminus\{0\})/\F_q^\times$ is the set of
projective denominator lines of coefficient length at most $D$.
Both $d(Q)$ and $m_r(Q)$ are invariant under multiplication of $Q$ by
an element of $\F_q^\times$, so they are well defined on a projective
denominator line $[Q]\in\mathbb P(V_D)$.  On the full-measure set
where every nonzero polynomial has finite common cancellation depth,
define, for $k\in\mathbb Z$, the exceedance count
\[
 \mathcal N_D(k)
 =
 \#\left\{
 [Q]\in\mathbb P(V_D):
 rm_r(Q)-d(Q)\ge \lfloor\log_qD\rfloor+k
 \right\}.
\]
Does the family of nested counts $\bigl(\mathcal N_D(k)\bigr)_{k\in
\mathbb Z}$ converge along subsequences for which the fractional part
of $\log_qD$ converges?  More specifically, is the limiting
exceedance process Poisson with an explicit intensity, and does this
imply a discrete-Gumbel law for
\[
        R_D^{(r)}-\lfloor\log_qD\rfloor?
\]
Projectivization removes the deterministic scalar clusters
$\{cQ:c\in\F_q^\times\}$, but dependencies among distinct denominator
lines and the lattice-valued centering still have to be controlled.
Such a result would complement the exact record identity in
Proposition~\ref{prop:record-duality} with a distributional law.

\subheading{Unequal-depth multiparameter renewal laws}
For
$\boldsymbol n=(n_1,\ldots,n_r)\in\Nzero^r$, define
\[
 L_r(\boldsymbol n)
 =
 \min\left\{
 d\ge1:
 \begin{array}{l}
 \text{there is }0\ne Q\in V_d\text{ such that}\\[-2pt]
 c(Q\alpha_i)\ge n_i\text{ for every }1\le i\le r
 \end{array}
 \right\}.
\]
A monotone nearest-neighbor path is a map
$\gamma:\Nzero\to\Nzero^r$ with $\gamma(0)=(0,\ldots,0)$ and
$\gamma(k+1)-\gamma(k)\in\{\boldsymbol e_1,\ldots,
\boldsymbol e_r\}$, where $\boldsymbol e_i$ is the $i$th standard
basis vector.  For which paths does the process
$k\mapsto L_r(\gamma(k))$ admit an exact fresh-innovation coding?
Can the pathwise codings be made compatible across different paths,
thereby producing a consistent multiparameter transition structure
on $\Nzero^r$?  Algorithms for deterministic multisequence synthesis
with unequal input lengths are already available
\cite{SchmidtSidorenko}; the open issue here is the exact Haar
probability law and the dependence between intersecting paths, not
the mere computation of $L_r(\boldsymbol n)$.

\section*{Acknowledgements}

This work was initiated during the author's visit to the School
of Mathematics at the Tata Institute of Fundamental Research.  The
author thanks Anish Ghosh and Gaurav Aggarwal for their
hospitality and stimulating conversations.
This work was supported by the Basic Science Research Program through
the National Research Foundation of Korea (NRF), funded by the Ministry
of Education (grant no.~RS-2025-25415913).

\end{document}